\documentclass[a4paper]{article}
\usepackage{amsmath}
\usepackage{amsfonts}
\usepackage{amssymb}
\usepackage{amsthm}
\usepackage{enumerate}
\usepackage[square, colon]{natbib}
\setcitestyle{aysep={},yysep={;}}
\theoremstyle{definition}

\theoremstyle{plain}
\newtheorem{proposition}{Proposition}
\newtheorem{lemma}[proposition]{Lemma}
\newtheorem{theorem}[proposition]{Theorem}

\theoremstyle{definition}

\usepackage{titling}
\title{Some Infinitary Paradoxes and Undecidable Sentences in Peano Arithmetic}
\author{Cheng Ka Yue\\[0.5em] chengkayue@gmail.com}
\date{}

\begin{document}
\maketitle 
\section{Introduction}

In \cite{Chaitin1995-Berry} there is a conversation between Gregory Chaitin and Kurt G\"odel:

\begin{quote}
[Chaitin] said, ``Professor G\"odel, I'm fascinated by your incompleteness theorem. I have a new proof based on Berry paradox that I'd like to tell you about." G\"odel said, ``It doesn't matter which paradox you use."
\end{quote}

To support this claim, we need to investigate what will happen if we formalize different paradoxes in Peano arithmetic (PA). Most notably, Chaitin proved a version of the First Incompleteness Theorem with a proof resembling the Berry paradox in his \cite{Chaitin1970-CHACCA}, so did George Boolos gave his proof using the same paradox (independently) in \cite{Boolos1989-Godel}.

In this paper\footnote{This paper was presented in Logic Colloquium 2015, and the results are from my mater's thesis \cite{Cheng2015thesis}.}, I will present a few infinitary paradoxes and corresponding undecidable sentences. The first three paradoxes are developed, in my master thesis, from a version of the Preface paradox, and the last one is an infinite version of the Surprise Examination paradox from \cite{Sorensen1993-SORTEU}.

We will work in the usual first order Peano arithmetic, though in fact the results hold in any theory that extends PA. The non-logical symbols in the language are the only constant symbol $0$, a unary function symbol $S$ and two binary function symbols $+$ and $\times$.

The technique being used to produce undecidable sentences in this paper, involving a general version of the Diagonal Lemma, is mainly from \cite{Cieslinski2013-CIEGTY}.

\section{Preliminaries}

In this section, I will state a few facts and definitions that are useful in this paper, proofs of those facts and details of the arithmetization of syntax will be skipped. These details can be found in books about G\"odel's Incompleteness Theorems, for examples, \cite{Smullyan1992-SMUGIT} and \cite{Smith2007-Intro}.

There are formulas in PA that are said to be  \emph{provable}. If a formula $\varphi$ is provable in Peano arithmetic, we will denote this fact by $\vdash \varphi$. Then we have some definitions:

\begin{enumerate}
\item A formula $\varphi$ is said to be \emph{refutable} if the negation of it, $\neg \varphi$, is provable.

\item A formula is decidable if it is provable or refutable, otherwise it is undecidable. Hence a formula $\varphi$ is undecidable if neither $\varphi$ nor $\neg \varphi$ is provable.

\item Two formulae $\varphi$ and $\psi$ are \emph{provable equivalent} if the formula $\varphi \longleftrightarrow \psi$ is provable.

\item A quantifier is \emph{bounded} in a formula if it is of the form $\exists x (x < t \land \varphi)$ or $\forall x (x < t \rightarrow \varphi)$, where $t$ is a term, and we will write $(\exists x < t)\varphi$ and $(\forall x <t) \varphi$ respectively.

\item A formula is a \emph{$\Delta_0$ formula} if it is provably equivalent to a formula containing only bounded quantifiers.

\item A formula is a \emph{$\Sigma_1$ formula} if it is provably equivalent to a formula of the form $\exists x \varphi$, where $\varphi$ is a $\Delta_0$ formula.
\end{enumerate}

We say that a theory is \emph{consistent} if there is no formula $\varphi$ such that both $\varphi$ and $\neg \varphi$ are provable. And we say that a theory is $\omega$-consistent if there is no open formula $\varphi(x)$ such that $\exists x \varphi(x)$ is provable, but for every natural number $n$, $\varphi(n)$ is not provable. In this paper we assume PA is both consistent and $\omega$-consistent.\footnote{The latter actually implies the former and can be replaced by weaker a weaker condition called 1-consistency (or $\Sigma_1$ soundness), which also implies consistency.}

The following corollary of the assumption of $\omega$-consistency is useful:
\begin{lemma} \label{omega-consist-lemma}
Let $\varphi(x)$ be a $\Sigma_1$ formula with a free variable $x$. If $\exists x \varphi(x)$ is provable, then there is a number $n$ such that $\varphi(n)$ is provable.\end{lemma}

This result simply follows from the definition of $\omega$-consistency and the fact that all $\Delta_0$ formulae are decidable.

Another lemma about $\Sigma_1$ formulae is also useful:
\begin{lemma} \label{sigma-1-lemma}
If $\varphi$ is a $\Sigma_1$ formula, then for any variable $x$, $\exists x \varphi$ is also a $\Sigma_1$ formula.
\end{lemma}

A proof of this lemma can be found in \cite{Smullyan1992-SMUGIT}.

We can encode each finite sequence of natural numbers into a natural number, call the \emph{code} of the sequence, in a way that we can also decode that number and obtain the original sequence. A number which is the code of a finite sequence is called a \emph{code number}.

Then we assign different numbers to the symbols in our object language, hence every expression corresponds to a finite sequence, which can be encoded into a natural number. Such a number is called the \emph{G\"odel number} of that expression. Let $\varphi$ be a formula, the G\"odel number of $\varphi$ will be denoted by $\ulcorner \varphi \urcorner$.

After that, (syntactical) properties and relations of expressions correspond to properties and relations of the G\"odel numbers of expressions. Then we can construct the following predicates and functions\footnote{Since we usually do not include function symbols other than $S$, $+$ and $\times$, ``$f(x)=y$" should be understood as a relation of $x$ and $y$ and ``$f(x)$" should be understood as pterms (p for pseudo) in \citep{Boolos1993-BOOTLO-3}.}:

\begin{enumerate}
\item $Code(x)$ is provable if $x$ is a code number.
\item $l(x)=n$ is provable if $x$ is a code number of a sequence with length $n$.
\item $Dec(x,k)=y$ is provable if $x$ is a code number and the $k^{th}$ term of the sequence encoded by $x$ is $y$.
\item $Neg(x)$ is a function such that $Neg(\ulcorner \varphi \urcorner)=\ulcorner \neg \varphi \urcorner$ is provable for any formula $varphi$.\footnote{In fact for any expressions, not restricted to formulas, but here we only concern formulas.}
\item $Subs(x,v,y)$ is a function such that if for any formula $\varphi$, term $t$, variable $v_i$ free in $\varphi$, then $Sub(\ulcorner \varphi \urcorner, \ulcorner v_i \urcorner, \ulcorner t \urcorner)=\ulcorner \varphi(t/v_i) \urcorner$, where $\varphi(t/v_i)$ is the formula obtained from substituting all free occurrence of $v_i$ in $\varphi$ by $t$, is provable.
\end{enumerate}

The above relations are $\Delta_0$. We also have an open $\Sigma_1$ formula $Prov(x)$ with one free variable satisfying the following two lemmas:

\begin{lemma}[$Prov$-introduction] \label{Prov-intro}
If $\varphi$ is a provable formula, then $Prov(\ulcorner \varphi \urcorner)$ is provable.
\end{lemma}

\begin{lemma}[$Prov$-elimination] \label{Prov-elim}
If Peano arithmetic is $\omega$-consistent, and $\varphi$ is a formula such that $Prov(\ulcorner \varphi \urcorner)$ is provable, then $\varphi$ is provable.
\end{lemma}

Since we assume the consistency and $\omega$-consistency of PA, $Prov(\ulcorner \varphi \urcorner)$ is provable if and only if $\varphi$ is provable for any formula $\varphi$.

Finally we need two more lemmas. The first one is a generalized version of the usual Diagonal Lemma, the proof of it can be found in \cite{Boolos1993-BOOTLO-3}:
\begin{lemma}[Generalized Diagonal Lemma] \label{Diag-lem}
Let $\varphi(x,y)$ be an open formula with two free variables $x,y$, then there is an open formula $\psi(x)$ with one free variable $x$ such that $\psi(x) \longleftrightarrow \varphi(x, \ulcorner \psi(x) \urcorner)$ is provable.
\end{lemma}

The second one is a consequence of G\"odel's Second Incompleteness Theorem:
\begin{lemma} \label{unprov-unprov}
Let $\varphi$ be a sentence, then $\neg Prov(\ulcorner \varphi \urcorner)$ is not provable.
\end{lemma}

\section{The Paradoxes}

In this section I will present four infinitary paradoxes, the first three of them are from my master thesis, though there are some similar finite version in the literature, I cannot find any name for the infinitary ones. The last one is called the Earliest Class Inspection paradox from \cite{Sorensen1993-SORTEU}, as noted in the introduction.

Imagine there are infinitely many people in a room, each of them say one and only one sentence. The following three situations correspond to the first three paradoxes.

\subsection*{Paradox 1: Someone is wrong.}

If everyone in the room says ``Someone is wrong" \footnote{To be more precise, it should be ``At least one sentence uttered in this room is false". But for convenience, we simply talk about people being right and wrong instead of the sentences they uttered is true or false.}, then it is impossible for everyone to be right, otherwise none of them is wrong, contradicting their claims. Hence someone must be wrong, but that person also says ``Someone is wrong", so ``No one is wrong" is true, contradicting him or her being wrong.

\subsection*{Paradox 2: Someone else is wrong.}

If everyone in the room says ``Someone else is wrong", then this situation is slightly more complicated. It is consistent that there is exactly one person being wrong, while the rest of them are right. Furthermore, since everyone says the same thing and it is a symmetric situation, it does not matter that which one is wrong. So the truth value assignments of their sentence is arbitrary in this sense, which is similar to the Truth-teller paradox, i.e. ``This sentence is true".

\subsection*{Paradox 3: Some people are wrong.}

Suppose all people in the room queue up, and the $k^{th}$ person says ``There are at least $k$ people wrong". Notice that if the $k^{th}$ person is right, then everyone before this person is also right.\footnote{Since the $k^{th}$ person is right, there are at least $k$ people wrong. Hence for any $j<k$, it is true that there are at least $j$ people wrong, which means the $j^{th}$ person is right.} Similarly, if the $k^{th}$ person is wrong, then everyone after this person is also wrong.

Using logic we know that either everyone is right or someone is wrong, in both cases the first person is right.\footnote{If everyone is right, then of course the first person is right; if someone is wrong, then at least one person is wrong, which is what the first person asserts.} Since the first person is right, someone must be wrong, and there must be someone who is the first person (in the queue) being wrong. Let this person be the $k^{th}$ person. By our observation everyone after her or him is wrong, so there are more than $k$ people wrong, but that means the $k^{th}$ person is right, and we have a contradiction.

\subsection*{Paradox 4: The Earliest Class Inspection Paradox}

Suppose you are a new teacher, and you are told that there will be a class inspection. There are two conditions on the date of the inspection: first, the sooner the better; second, you do not know and cannot guess the day so that you cannot prepare for it. Therefore, the inspection will be on the first day which you do not believe there will be a class inspection.

Now, the next school day is the first available day for class inspection, but then the inspection cannot be on that day since you can reason it. Similarly you can rule out the possibilities of the inspection being on the second day, the third day, the fourth day, and so on. Hence the earliest unexpected class inspection is impossible.

\subsection*{Notes on the paradoxes}

Here are some notes on the paradoxes.

The finite version of the first paradox, which is still a paradox, is related to the Liar cycles.\footnote{That is, the people are in a circle, and everyone says the next person is wrong, with an exception that if the number of people is even, then one of them says the next one is right.} However an existential quantifier is informally used in this paradox, and it can be extended to the infinite case easily, unlike the Liar cycles.

The finite version of the second paradox is again similar to the infinite case, that is, we have different consistent truth-value assignments. It is related to a paradox by Jean Buridan, which is ``Socrates says that Plato tells a lie, Plato says that Socrates tells a lie" (and they say nothing more), it is also called the No-No paradox in \cite{Sorensen2004}. In Sorensen's book, there is a finite version of the second paradox.\footnote{In almost the same form, except that he uses a list of 100 sentences, each of them is the sentence ``Some other sentence on this list is false".}

The finite version of the third paradox is not necessarily a paradox: if the number of people is even, then the first half people are right and the second half wrong; if the number of people is odd, then the one who is in the exact middle of the queue is in a Liar paradox situation. And the infinite version resembles the Yablo's paradox\footnote{Yablo's paradox is about an infinite list of sentences where every sentence is ``All sentences below are false". If the first sentence is true, then the second one is false, hence some sentence below it is true, contradicting the first one. On the other hand, if the first sentence is false, then some sentence below it is true, and we will get a similar contradiction.}

The Earliest Class Inspection Paradox can be regarded as an infinite version of the Surprise Examination Paradox.

\section{The Undecidable Sentences}

In \cite{Cieslinski2013-CIEGTY}, the authors apply Lemma \ref{Diag-lem} to an open formula to obtain undecidable sentences resembling Yablo's paradox. In the following, we will do the same for the paradoxes in the last section.

\subsection*{Formalizing Paradox 1}

Consider the open formula\footnote{The subformula $0 \leq x$ is to make sure the formula contains $x$ as a free variable. It is not difficult to prove that $0 \leq x$ is provable.}:
\begin{equation*}
\exists z \Big(Prov \big(Neg(Subs(y, \ulcorner x \urcorner, \ulcorner z \urcorner))\big) \Big) \land (0 \leq x)
\end{equation*}
By Lemma \ref{Diag-lem} there is an open formula $\mathbf{P}(x)$ with one free variable $x$ such that:
\begin{align*}
& \vdash \mathbf{P}(x) \longleftrightarrow \exists z \Big( Prov\big(Neg(Subs(\ulcorner \mathbf{P}(x) \urcorner, \ulcorner x \urcorner, \ulcorner z \urcorner)) \big) \Big) \land (0 \leq x) \\
\Rightarrow \, & \vdash \mathbf{P}(x) \longleftrightarrow \exists z \big(Prov(\ulcorner \neg \mathbf{P}(z) \urcorner) \big) \land (0 \leq y) \\
\Rightarrow \, & \vdash \mathbf{P}(x) \longleftrightarrow \exists z \big(Prov(\ulcorner \neg \mathbf{P}(z) \urcorner) \big)
\end{align*}

Here $\mathbf{P}(k)$ can be read as the sentence that the $k^{th}$ person says, which is, intuitively, ``There is someone whose sentence is refutable". Like the undecidable sentence in G\"odel original proof, which is a formalization of the Liar paradox, we replace ``truth" by ``provability", since the latter can be formulated in the object language.

We have the following result:

\begin{theorem}
For any natural number $k$, $\mathbf{P}(k)$ is undecidable.
\end{theorem}

\begin{proof}
Let $k$ be a natural number. Suppose $\mathbf{P}(k)$ is provable. Then by the choice of $\mathbf{P}(x)$, $\exists z \big(Prov(\ulcorner \neg \mathbf{P}(z) \urcorner) \big)$ is also provable. By $\omega$-consistency of PA, Lemma \ref{omega-consist-lemma} and Lemma \ref{sigma-1-lemma}, there is a natural number $n$ such that $\neg Prov(\ulcorner \mathbf{P}(n) \urcorner)$ is provable. But by Lemma \ref{unprov-unprov} this is impossible.

On the other hand, suppose $\mathbf{P}(k)$ is refutable. Then:
\begin{align*}
\vdash \neg \mathbf{P}(k) \Rightarrow \, & \vdash \neg \exists z \big( Prov(\ulcorner \neg \mathbf{P}(z) \urcorner) \big) & \text{(By the choice of $\mathbf{P}(x)$)} \\
\Rightarrow \, & \vdash \forall z \neg Prov(\ulcorner \neg \mathbf{P}(z) \urcorner)\\
\Rightarrow \, & \vdash \neg Prov(\ulcorner \neg \mathbf{P}(k) \urcorner) & \text{(By universal instantiation)}
\end{align*}

While by Lemma \ref{Prov-intro}, we have $Prov(\ulcorner \neg \mathbf{P}(k) \urcorner)$ provable. Since we assume that PA is consistent, this is impossible.

Therefore for any natural number $k$, $\mathbf{P}(k)$ is neither provable nor refutable, hence undecidable.
\end{proof}

\subsection*{Formalizing Paradox 2}

Consider the following open formula:
\begin{equation*}
\exists z \Big(z \neq x \land Prov\big( Neg(Subs(y, \ulcorner x \urcorner, \ulcorner z \urcorner)) \big) \Big)
\end{equation*}
By Lemma \ref{Diag-lem}, there is an open formula $\mathbf{Q}(x)$ with one free variable $x$ such that
\begin{equation*}
\vdash \mathbf{Q}(x) \longleftrightarrow \exists z \big(z \neq x \land Prov(\ulcorner \neg \mathbf{Q}(z) \urcorner) \big)
\end{equation*}

Similar to the previous formalization, $\mathbf{Q}(k)$ can be read as the sentence that the $k^{th}$ person says, which is, intuitively, ``There is someone else whose sentence is refutable".

We have the following result:
\begin{theorem}
For any natural number $k$, $\mathbf{Q}(k)$ is undecidable. 
\end{theorem}

\begin{proof}
Let $k$ be a natural number. Suppose $\mathbf{Q}(k)$ is refutable. Then:
\begin{align*}
\vdash \neg \mathbf{Q}(k) \Rightarrow \, & \vdash \neg \exists z \big( z \neq k \land Prov(\ulcorner \neg \mathbf{Q}(z) \urcorner) \big) & \text{(By the choice of $\mathbf{Q}(x)$)} \\
\Rightarrow \, & \vdash \forall z \big( Prov(\ulcorner \neg \mathbf{Q}(z) \urcorner)  \rightarrow z=k \big) \\
\Rightarrow \, & \vdash \forall z \big( z \neq k \rightarrow \neg Prov( \ulcorner \neg \mathbf{Q}(z) \urcorner)\big) \\
\Rightarrow \, & \vdash \big( k+1 \neq k \rightarrow \neg Prov(\ulcorner \neg \mathbf{Q}(k+1) \urcorner) \big) & \text{(By Universal Instantiation)}\\
\Rightarrow \, & \vdash \neg Prov(\ulcorner \neg \mathbf{Q}(k+1) \urcorner) & \text{($\vdash k+1 \neq k$)}
\end{align*}

But by Lemma \ref{unprov-unprov}, $\neg Prov(\ulcorner \neg \mathbf{Q}(k+1) \urcorner)$ is not provable. Hence $\mathbf{Q}(k)$ cannot be refutable.

On the other hand, suppose $\mathbf{Q}(k)$ is provable, then by the choice of $\mathbf{Q}(x)$, $\exists z \big( z \neq k \land Prov(\ulcorner \neg \mathbf{Q}(z) \urcorner) \big)$ is also provable. By $\omega$-consistency of PA, Lemma \ref{sigma-1-lemma} and Lemma \ref{omega-consist-lemma}, there is a natural number $n$ such that $Prov(\ulcorner \neg \mathbf{Q}(n) \urcorner)$ is provable.

By Lemma \ref{Prov-elim}, $\neg \mathbf{Q}(n)$ is also provable. But this contradicts the first half of this proof, hence $\mathbf{Q}(k)$ cannot be provable.

Therefore for any natural number $k$, $\mathbf{Q}(k)$ is undecidable.
\end{proof}

We have noted the similarity between paradox 2 and the Truth-teller paradox. Nevertheless, the Henkin sentence, the formalized Truth-teller, is provable by L\"ob's celebrated theorem, while the formalized version of paradox 2 above is undecidable.

\subsection*{Formalizing Paradox 3}

To formalize the third paradox, it is more complicated since we need to refer to a set of numbers (of size $k$) in the object language. So we need a two more definitions:

\begin{itemize}
\item $HetSeq(x) \longleftrightarrow Code(x) \land \big( \forall y \leq l(x) \big) \big( \forall z \leq l(x) \big) \big(y \neq z \rightarrow Dec(y,x) \neq Dec(z,x) \big)$ \\
If $HetSeq(x)$ is provable, then $x$ is a code number of a sequence where no two terms are the same.
\item $Ele(x,y) \longleftrightarrow Code(y) \land \big(\exists u \leq l(y) \big) \big( Dec(u, y)=x \big)$ \\ If $Ele(x,y)$ is provable, then $x$ represents a number which is a term of the sequence represented by $y$.
\end{itemize}
The idea is to define a kind of sequence, called heterosequence, in which no two terms are the same. Instead of saying there is a set of $k$ natural numbers, we can say there is a heterosequence of length $k$. Note that both $HetSeq(x)$ and $Ele(x,y)$ are $\Delta_0$.

Then consider the open formula:
\begin{equation*}
\exists z \Big[HetSeq(z) \land l(z)=Sx \land (\forall t \leq z) \big[ Ele(t,z) \rightarrow Prov\big( Neg(Subs(y, \ulcorner x \urcorner, \ulcorner t \urcorner))\big) \big] \Big]
\end{equation*}
Again by Lemma \ref{Diag-lem} there is an open formula $\mathbf{R}(x)$ such that:
\begin{equation*}
\vdash \mathbf{R}(x) \longleftrightarrow \exists z \Big(HetSeq(z) \land \big( l(z)=Sx \big) \land (\forall t \leq z) \big( Ele(t,z) \rightarrow  Prov(\ulcorner \neg \mathbf{R}(t) \urcorner) \big) \Big)
\end{equation*}

Intuitively, $\mathbf{R}(k)$ is provable if and only if there is a heterosequence of length $k+1$ such that for each element $t$ of that sequence, $\mathbf{R}(t)$ is refutable.\footnote{The length of the heterosequence is $k+1$, since we count the natural numbers from $0$.}

We have the following two lemmas:
\begin{lemma} \label{backward-heredity}
If $m, n$ are natural numbers and $m<n$, then $\vdash \mathbf{R}(n) \rightarrow \mathbf{R}(m)$.
\end{lemma}

\begin{proof}
Let $m,n$ be natural numbers and $m<n$. Suppose $\vdash \mathbf{R}(n)$. Then by the choice of $\mathbf{R}(x)$:

$\vdash \exists z \Big(HetSeq(z) \land \big(l(z)= Sn \big) \land (\forall t \leq z) \big( Ele(t,z) \rightarrow Prov(\ulcorner \neg  \mathbf{R}(t) \urcorner) \big) \Big)$

By $\omega$-consistency, there is a natural number $N$ such that

$\vdash \Big(HetSeq(N) \land \big(l(N)= Sn \big) \land (\forall t \leq N) \big( Ele(t,N) \rightarrow Prov(\ulcorner \neg \mathbf{R}(t) \urcorner) \big) \Big)$

$N$ is the code number of a heterosequence of length $n+1$, then we can take the first $m+1$ terms of the sequence to form a new heterosequence, and let its code number be $M$. By definition, both $l(b)= Sm$ and $Ele(t,M) \rightarrow Ele(t,N)$ are provable, hence

$\vdash \Big(HetSeq(M) \land \big(l(b)= Sm \big) \land (\forall t \leq M) \big( Ele(t,M) \rightarrow Prov(\ulcorner \neg \mathbf{R}(t) \urcorner) \big) \Big)$

Therefore $\vdash \mathbf{R}(m)$. By the deduction theorem we get $\vdash \mathbf{R}(n) \rightarrow \mathbf{R}(m)$.
\end{proof}

\begin{lemma} \label{hereditarily-refutable}
If $m, n$ are natural numbers and $m>n$, then $\vdash \neg \mathbf{R}(n) \rightarrow \neg \mathbf{R}(m)$.
\end{lemma}
\begin{proof}
By Lemma \ref{backward-heredity} we have $\vdash \mathbf{R}(m) \rightarrow \mathbf{R}(n)$, which implies the contrapositive of the formula, therefore $\vdash \neg \mathbf{R}(n) \rightarrow \neg \mathbf{R}(m)$.
\end{proof}

These two lemmas formalize our previous observations in the situation of paradox 3: ``if the $k^{th}$ person is right then everyone before him or her is right" and ``if the $k^{th}$ person is wrong then everyone after her or him is wrong".

Then we have the following result:

\begin{theorem}
For any natural number $n$, $\mathbf{R}(n)$ is undecidable.
\end{theorem}

\begin{proof}
Let $n$ be a natural number. Suppose $\mathbf{R}(n)$ is refutable, then $\vdash \neg \mathbf{R}(n)$.

By Lemma \ref{hereditarily-refutable}, for every $m>n$, $\mathbf{R}(m)$ is refutable. Therefore the sentences $\neg \mathbf{R}(n), \neg \mathbf{R}(n+1), \ldots, \neg \mathbf{R}(n+n)$ are all provable, by Lemma \ref{Prov-intro} the sentences $Prov(\ulcorner \neg \mathbf{R}(n) \urcorner), Prov(\ulcorner \neg \mathbf{R}(n+1) \urcorner), \ldots, Prov(\ulcorner \neg \mathbf{R}(n+n) \urcorner)$ are also provable.

Let $c$ be the code number of the sequence $(n,n+1, \ldots, n+n)$. Then the sentences $HetSeq(c)$, $l(c)=Sn$, and $(\forall t \leq c) \big( Ele(t,c) \rightarrow Prov(\ulcorner \neg \mathbf{R}(t) \urcorner) \big)$ are all provable. This implies that $\mathbf{R}(n)$ is provable and we get another contradiction.

On the other hand, suppose $\mathbf{R}(n)$ is provable. Then
\begin{equation*}
\exists z \Big(HetSeq(z) \land \big(l(z)= Sn\big) \land (\forall t \leq z) \big( Ele(t,z) \rightarrow Prov(\ulcorner \neg \mathbf{R}(t) \urcorner) \big) \Big)
\end{equation*}
is also provable. By $\omega$-consistency, there is an number $c$ such that
\begin{align*}
& \vdash \Big(HetSeq(c) \land l(c)= Sn \land (\forall t \leq c) \big( Ele(t,c) \rightarrow Prov(\ulcorner \neg \mathbf{R}(t) \urcorner) \big) \Big) \\
\Rightarrow \,& \vdash (\forall t \leq c) \big( Ele(t,c) \rightarrow Prov(\ulcorner \neg \mathbf{R}(t) \urcorner))\big) \Big) \quad \qquad  \text{(By conjunction elimination)} \\
\Rightarrow \,& \vdash Prov (\ulcorner \neg \mathbf{R}(Dec(1,c) \urcorner) \qquad \qquad \qquad \qquad \qquad \qquad \qquad \qquad \qquad \, \text{(By fact 3)}
\end{align*}
But it is impossible by the first half of this proof, so $\mathbf{R}(n)$ is not provable. Therefore $\mathbf{R}(n)$ is neither provable nor refutable.
\end{proof}

\subsection*{Formalizing Paradox 4}
Consider the open formula:
\begin{equation*}
(\forall z < x) Prov \big(Neg(Subs(y,\ulcorner x \urcorner, \ulcorner z \urcorner))\big) \rightarrow Prov \big(Neg(y)\big)
\end{equation*}

Apply Lemma \ref{Diag-lem}, we will get an open formula $\mathbf{F}(x)$ with one variable $x$ such that:
\begin{equation*}
\vdash \mathbf{F}(x) \longleftrightarrow \big[ (\forall z < x) Prov \big(\ulcorner \neg \mathbf{F}(z) \urcorner \big) \rightarrow Prov \big(\ulcorner \neg \mathbf{F}(x) \urcorner) \big) \big] 
\end{equation*}

Roughly speaking, $\mathbf{F}(n)$ is related to the proposition ``there will be a class inspection at the $(n+1)^{st}$ day"\footnote{Note that we count from $0$.}. And it satisfies the condition that if it is deducible that there is no class inspection at the first $n$ day, then it is deducible that there is no class inspection at the $(n+1)^{st}$ day.

Then we have the following result:
\begin{theorem} \label{Surprise}
$\exists x \mathbf{F}(x)$ is undecidable.
\end{theorem}

To prove this theorem, we need the following fact:
\begin{proposition}[The Least Number Principle]
For any open formula $P(x)$ with exactly one free variable $x$, it is provable that $\exists x P(x) \rightarrow \exists x \big( P(x) \land (\forall y<x) \neg P(y) \big)$.
\end{proposition}

The proof of this principle, which can be found in \cite{Boolos1993-BOOTLO-3}, is skipped here. Now we can prove Theorem \ref{Surprise}.

\begin{proof}
Suppose $\exists x \mathbf{F}(x)$ is provable, then by the Least Number Principle, $\exists x \big( \mathbf{F}(x) \land (\forall z<x) \neg \mathbf{F}(z) \big)$ is also provable.

By the $\omega$-consistency, there is a natural number $n$ such that the formula $(\forall z < n) \neg \mathbf{F}(z) \land \mathbf{F}(n)$ is provable. Then $(\forall z < n) \neg \mathbf{F}(z)$ is provable, by substitution and modus ponens we have $\neg \mathbf{F}(0), \neg \mathbf{F}(1), \ldots, \neg \mathbf{F}(n-1)$ are all provable. Hence by Lemma \ref{Prov-intro}, $Prov(\ulcorner \neg \mathbf{F}(0) \urcorner), Prov(\ulcorner \neg \mathbf{F}(1) \urcorner), \ldots, Prov(\ulcorner \neg \mathbf{F}(n-1) \urcorner)$ are all provable, so is $(\forall z < n) Prov(\ulcorner \neg \mathbf{F}(z) \urcorner)$.

Since $\mathbf{F}(n)$ is also provable, we have:
\begin{align*}
\vdash \mathbf{F}(n) \Rightarrow \, & \vdash \big[ (\forall z < n) Prov \big(\ulcorner \neg \mathbf{F}(z) \urcorner \big) \rightarrow Prov \big(\ulcorner \neg \mathbf{F}(n) \urcorner) \big) \big] & \text{(By the choie of $\mathbf{F}(x)$)} \\
\Rightarrow \, & \vdash Prov \big(\ulcorner \neg \mathbf{F}(n) \big) & \text{(By modus ponens)} \\
\Rightarrow \, & \vdash \neg \mathbf{F}(n) & \text{(By Lemma \ref{Prov-elim})}
\end{align*}
So we get a contradiction, and $\exists x \mathbf{F}(x)$ is not provable.

On the other hand, suppose $\exists x \mathbf{F}(x)$ is refutable. Then $\neg \exists x \mathbf{F}(x)$, and equivalently, $\forall x \neg \mathbf{F}(x)$ are provable. For any natural number $n$, we have:
\begin{align*}
& \vdash \neg \mathbf{F}(n) & \text{(By universal instantiation)} \\
\Rightarrow \, & \vdash (\forall z < n)Prov \big( \ulcorner \neg \mathbf{F}(z) \urcorner \big) \land \neg Prov(\ulcorner \neg \mathbf{F}(n) \urcorner) & \text{(By the choice of $\mathbf{F}(x)$)} \\
\Rightarrow \, & \vdash \neg Prov(\ulcorner \neg \mathbf{F}(n) \urcorner) & \text{(By conjunction elimination)}
\end{align*}
But by Lemma \ref{Prov-intro}, $\neg \mathbf{F}(n)$ is provable implies that $Prov(\ulcorner \neg \mathbf{F}(n) \urcorner)$ is also provable. Again we get a contradiction.

Therefore, $\exists x \mathbf{F}(x)$ is undecidable.
\end{proof}

With some modifications on the definition of $\mathbf{F}(x)$, we can obtain a formalization of the Surprise Examination paradox which is essentially different from the one in \cite{Fitch1964-FITAGF}, since the former is undecidable but the latter is refutable.

\section{Summary}

We have seen four infinitary paradoxes, and four related open formulas, $\mathbf{P}(x)$, $\mathbf{Q}(x)$, $\mathbf{R}(x)$, and $\mathbf{F}(x)$. The first three open formulas lead to infinitely many undecidable sentences, and for the last one we have an undecidable sentence $\exists x \mathbf{F}(x)$.

These results partly confirm G\"odel's claim quoted in the first section, and refute a possible counterexample from \cite{Fitch1964-FITAGF}. Also, it is interesting to investigate whether there is any other paradox like the paradox 2, and to understand the difference between paradox 2 and the Truth-teller.

\bibliography{bib}
\bibliographystyle{apalike}
\end{document}